\documentclass[titlepage,12pt]{article} 
\usepackage{hyperref}
\usepackage[usenames,dvipsnames]{xcolor}
\usepackage[title]{appendix}
\usepackage{amssymb,amsthm,amsmath} 
\usepackage[a4paper]{geometry}
\usepackage{datetime2}
\usepackage{enumitem}
\usepackage[utf8]{inputenc}

\date{}

\newcommand{\ep}{\varepsilon}

\newcommand{\re}{\mathbb{R}}

\newtheorem{thm}{Theorem}[section]

\newtheorem{rmk}[thm]{Remark}
\newtheorem{prop}[thm]{Proposition}
\newtheorem{defn}[thm]{Definition}

\newtheorem{lemma}[thm]{Lemma}
\newtheorem{open}{Open problem}

\title{{\bf Universal bounds for a class of second order evolution equations and applications}}
\author{Marina Ghisi\vspace{1ex}\\ 
{\normalsize Universit\`a degli Studi di Pisa} \\
{\normalsize Dipartimento di Matematica}\\ 
{\normalsize PISA (Italy)}\\
{\normalsize e-mail: \texttt{marina.ghisi@unipi.it}}
\and
Massimo Gobbino\vspace{1ex}\\ 
{\normalsize Universit\`a degli Studi di Pisa} \\
{\normalsize Dipartimento di Ingegneria Civile e Industriale}\\ 
{\normalsize PISA (Italy)}\\  
{\normalsize e-mail: \texttt{massimo.gobbino@unipi.it}}
\and
Alain Haraux\vspace{1ex}\\ 
{\normalsize Sorbonne Universit\'e, Universit\'e Paris-Diderot SPC, CNRS, INRIA}, \\
{\normalsize Laboratoire Jacques-Louis Lions,  LJLL, F-75005,
Paris, France.}\\ 
{\normalsize e-mail: \texttt{haraux@ann.jussieu.fr}}}%\pagestyle{plain}

\numberwithin{equation}{section}

\begin{document}\maketitle

\begin{abstract} 

We consider a class of abstract second order evolution equations with a restoring force that is strictly superlinear at infinity with respect to the position, and a dissipation mechanism that is strictly superlinear at infinity with respect to the velocity.

Under the assumption that the growth of the restoring force dominates the growth of the dissipation, we prove a universal bound property, namely that the energy of solutions is bounded for positive times, independently of the initial condition.  Under a slightly stronger assumption, we show also a universal decay property, namely that the energy decays (as time goes to infinity) at least as a multiple of a negative power of $t$, again independent of the boundary conditions.

We apply the abstract results to solutions of some nonlinear wave, plate and Kirchhoff equations in a bounded domain. 
\vspace{6ex}

\noindent{\bf Mathematics Subject Classification 2010 (MSC2010):} 
35B40, 35L70, 35L75, 35Q74, 35L90.

% 35B40   	Asymptotic behavior of solutions
% 35B36   	Pattern formation
% 35L70   	Nonlinear second-order hyperbolic equations
% 35Q74   	PDEs in connection with mechanics of deformable solids
% 35L75   	Nonlinear higher-order hyperbolic equations
% 35L35   	Initial-boundary value problems for higher-order hyperbolic equations
% 35L90   	Abstract hyperbolic equations
		
\vspace{6ex}


\end{abstract}

\section {Introduction}
 
It is well known that a certain number of dynamical systems $S(t)$ defined on a Banach space $X$ have the property of universal boundedness for all $t>0$, in the sense that 
\begin{equation*} 
\forall t>0,
\qquad
S(t) X \hbox { is a bounded subset of } X.
\end{equation*}  

As a simple example we can consider the first order scalar ordinary differential equation 
\begin{equation*}  
u' + \delta |u|^{\rho} u = 0, 
\end{equation*}  
where $\delta$ and $\rho$ are positive real numbers. Indeed, integrating this differential equation we find that 
%\begin{equation*}
%\forall t>0, 
%\qquad 
%|u(t)|\leq\frac {1}{(\rho\delta)^{1/\rho}}\cdot\frac {1}{t^{1/\rho}},
%\end{equation*}
\begin{equation*}
|u(t)|\leq\frac {1}{(\rho\delta)^{1/\rho}}\cdot\frac {1}{t^{1/\rho}}
\qquad
\forall t>0, 
\end{equation*}
independently of the initial condition $u(0)$. This property extends classically to some classes of nonlinear parabolic partial differential equations, for instance the semilinear parabolic equation such as
$$u_{t}-\Delta u + \delta |u|^{\rho} u = 0$$  
with either Dirichlet or Neumann homogeneous boundary conditions. The result follows at once from the maximum principle. We refer to~\cite{1975-SNS-Simon} for a more elaborate quasilinear case.
\bigskip

It is natural to ask whether an analogous universal bound property holds true for second order ordinary differential equations with superlinear dampings such as  
\begin{equation*}  
u''+ \omega^2u + \delta |u'|^{\rho} u' = 0. 
\end{equation*}

The answer is in general negative. Indeed, in the special case $\omega=\rho=\delta=1$, this equation has a solution with explicit expression $u(t)=t^{2}/4-1/2$ for $t\leq 0$ that extends uniquely for $t\geq 0$. Due to the autonomous character of the equation, the entire range of this unbounded solution is contained in $S(t)\re^2$ for all positive times, and hence the universal boundedness fails.

The next step  is to consider scalar second order equations with both superlinear damping and superlinear restoring force, such as
\begin{equation}
u'' + |u'|^{\alpha}u' + |u|^\beta u=0.
\label{ode:souplet}
\end{equation} 

For this equation, P.~Souplet~\cite{1998-DIE-Souplet} gave a definitive negative answer in the regime $\alpha\ge \beta\ge 0$. On the other hand, in the regime $0< \alpha < \beta$, it was shown very recently in~\cite{A-H} that the universal boundedness holds. More precisely, if we consider the classical energy
\begin{equation}
%\label{energode}
E(t) = \frac{1}{2}u'(t)^{2} + \frac{1}{\beta +2}|u(t)|^{\beta +2}, 
\nonumber
\end{equation}
the method of \cite{A-H}  yields the optimal estimate      
\begin{equation}
%\label{optode} 
E(t)\leq C \max \left\{t ^{-2/\alpha},t^{-(\alpha+1)(\beta + 2)/(\beta -\alpha)}\right\}
\qquad
\forall t>0,
\nonumber
\end{equation} 
where $C$ does not depend on the initial data. 

After this result, which can be easily extended in the finite dimensional vector framework,  it is reasonable to ask what happens in the regime $0<\alpha<\beta$ for wave equations such as
\begin{equation}
\label{wave}
u_{tt}  -\Delta u + |u|^\beta u + |u_{t}|^{\alpha}u_{t}=0, 
\end{equation} 
with either Dirichlet or Neumann boundary conditions, or for analogous plate equations where the Laplacian is replaced by a bi-Laplacian. The issue seems to be non-trivial because there is no such maximum principle as in the parabolic case, and because without the nonlinear term $|u|^\beta u$ the universal boundedness does not take place (see~\cite{1994-JMPA-Carpio}).\bigskip

Nevertheless, a natural slight change of the method of~\cite{A-H}, inspired by a technique devised in~\cite{Rio} involving a power of the total energy, gives the result for a large class of equations that fit into a natural functional framework. For these equations, we prove in Theorem~\ref{thm:ubp} the universal bound for all positive times, and in Theorem~\ref{thm:ud1} and Theorem~\ref{thm:ud2} the universal decay at infinity under slightly stronger assumptions. It should be even possible to extend the universal bound and universal decay properties to some singular equations and systems such as those studied in~\cite{AH, A^2H}, at least in the finite dimensional case.
\bigskip

The plan of the paper is as follows. In section~\ref{sec:functional} we introduce the general functional setting and we list the assumptions that we need in the sequel. In section~\ref{sec:abstract} we state and prove our abstract results. In section~\ref{sec:pde} we show some examples of application of our abstract theory to partial differential equations. Finally, in section~\ref{sec:open} we present some negative results and open problems. 
 
%\clearpage

\section{Functional setting}\label{sec:functional}

Let $H$ be a Hilbert space, and let $V$ be another Hilbert space continuously imbedded into $H$ as a dense subspace. If we identify $H$ with its dual $H'$, we obtain a classical Hilbert triple $V\subseteq H\subseteq V'$. We denote norms by double bars, and scalar products and duality pairings by angle brackets.

Let $T>0$ be a real number. In the time interval $[0,T]$ or $(0,T)$ we consider evolution equations of the form
\begin{equation}
u''(t)+\nabla F(u(t))+g(t,u'(t))=0,
\label{eqn:abstract}
\end{equation}
where $F$ and $g$ satisfy the following assumptions.

\begin{itemize}

\item[(F1)]
The function $F:V\to\re$ is of class $C^{1}$, and $\nabla F\in C^{0}(V,V')$ is its gradient.

\item[(G1)]
The function $g:(0,T)\times V\to V'$ is such that for every $v\in L^{\infty}((0,T),V)$ the function $t\to g(t,v(t))$ belongs to $L^{1}((0,T),V')$.

\end{itemize}

Under these assumptions we can introduce a notion of strong solutions to (\ref{eqn:abstract}).
\begin{defn}[Strong solutions]\label{defn:strongsol} 
\begin{em}

A \emph{strong solution} to (\ref{eqn:abstract}) is a function
\begin{equation}
u\in W^{1,\infty}((0,T),V)\cap W^{2,1}((0,T),V')
\nonumber
\end{equation}
for which (\ref{eqn:abstract}) holds true as an equality in $L^{1}((0,T),V')$.

\end{em}
\end{defn}

\begin{rmk}
\begin{em}

Every strong solution belongs in particular to the class
$$C^{0}([0, T],V)\cap C^{1}([0,T],H),$$
and as a consequence the pointwise values $u(t)\in V$ and $u'(t)\in H$ are well defined for every $t\in[0,T]$ (endpoints included). Moreover, the classical energy
\begin{equation}
E_0(t):=\frac{1}{2}\|u'(t)\|_{H}^{2}+F(u(t))
\label{defn:E}
\end{equation}
belongs to $W^{1,\infty}((0,T),\re)$ and
\begin{equation}
E_0'(t)=-\langle g(t,u'(t)),u'(t)\rangle_{V',V}
\label{eqn:E'}
\end{equation}
for almost every $t\in(0,T)$.

\end{em}
\end{rmk} 

Now we assume that $\alpha$ and $\beta$ are two positive real numbers, and that $X$ and $Y$ are two Banach spaces that extend the original Hilbert triple $V\subseteq H\subseteq V'$ to a chain of seven spaces with continuous imbeddings
\begin{equation}
V\subseteq Y\subseteq X\subseteq H\subseteq X'\subseteq Y'\subseteq V'.
\label{7-spaces}
\end{equation}

The following additional assumptions on $F$ and $g$ are needed in our abstract result concerning the uniform bound property (see Theorem~\ref{thm:ubp}).
\begin{itemize}

\item[(F2)]
There exist real numbers $\delta_{1}>0$ and $C_{1}\geq 0$ such that
\begin{equation}
F(u)\geq\delta_{1}\|u\|_Y^{\beta+2}-C_1, 
\qquad
\forall u\in V.
\label{hp:F2}
\end{equation}

\item[(F3)]
There exist real numbers $\delta_{2}>0$ and $C_{2}\geq 0$ such that
\begin{equation} 
\langle \nabla F(u), u \rangle_{V',V} \geq \delta_{2} F(u) -C_2 .
\qquad
\forall u\in V,
\label{hp:F3}
\end{equation}

\item[(G2)]
There exist real numbers $\delta_{3}>0$ and $C_{3}\geq 0$ such that
\begin{equation}
\langle g(t, v), v\rangle_{V', V} \geq \delta_{3}\|v\|_X^{\alpha+2}-C_3
%\qquad
%\forall (t,v)\in (0,T)\times V.
\label{hp:G2}
\end{equation}
for every $v\in V$ and almost every $t\in(0,T)$.

\item[(G3)] For every $(t,v)\in(0,T)\times V$ it turns out that $g(t,v)\in X'$, and there exist  real numbers $C_{4}\geq 0$ and $D_{4}>0$ such that
\begin{equation}
\|g(t, v)\|_{X'}\leq{D_4}\|v\|_X^{\alpha+1}+C_{4}
%\qquad
%\forall (t,v)\in (0,T)\times V.
\label{hp:G3}
\end{equation}
for every $v\in V$ and almost every $t\in(0,T)$.

\end{itemize}

The assumptions above are needed, in the simplified version with all $C_{i}$'s equal to zero, also in our first version of the uniform decay property (see Theorem~\ref{thm:ud1}). In our second version of the uniform decay property (see Theorem~\ref{thm:ud2}) we need two further assumptions, namely that there exists a real number ${C_5}>0$ such that
\begin{equation}
\|u\|_{X}^{\alpha+2}\leq{C_5}\left(\|u\|_{H}^{2}+\|u\|_{Y}^{\beta+2}\right)
\qquad
\forall u\in V,
\label{hp:norms}
\end{equation} 
and that we can estimate $F(u)$ from below as follows.
\begin{itemize}

\item[(F4)]  There exists a real number $\delta_{4}>0$ such that
\begin{equation}
F(u)\geq\delta_{4}\|u\|_{H}^{2}
\qquad
\forall u\in V.
\label{hp:F4}
\end{equation}

\end{itemize}

\begin{rmk}
\begin{em}

The notion of strong solutions, the computation of the time-derivative of the energy in (\ref{eqn:E'}), as well as all the assumptions introduced above, can be extended in a standard way to solutions defined in the half-line $[0,+\infty)$ or $(0,+\infty)$.
\end{em}
\end{rmk} 

%\clearpage

\section{Abstract results}\label{sec:abstract}

\subsection{Universal bound for all positive times}%\label{sec:ubp}

In this section we prove that, for all positive times, solutions lie in a bounded subset of the phase space. We shall use the  following simple  universal bound property for a class of  differential inequalities which was stated and proved in~\cite[Lemma~III.5.1]{Temam}, with a reference to J.M.~Ghidaglia.

\begin{lemma}\label{lemma:ubp}

Let $T>0$ be a positive real number, and let $\Phi\in W^{1,\infty}((0,T),\re)$ be a nonnegative function. Let us assume that there exist positive real numbers $\rho$, $\gamma$ and $M$ such that
\begin{equation}
\Phi'(t)\leq -\rho\,\Phi(t)^{1+\gamma}+M
\nonumber
\end{equation}
for almost every $t\in (0,T)$.Then we have 
\begin{equation}
\Phi(t)\le\left(\frac{1}{\gamma\rho t}\right)^\frac{1}{\gamma} + \left(\frac{M}{\rho}\right)^\frac{1}{1+\gamma} 
\qquad
\forall t\in(0,T).
\nonumber
\end{equation}
\end{lemma} \begin{proof} This follows from the standard comparison principle since the function $$ \Psi(t)=\left(\frac{1}{\gamma\rho t}\right)^\frac{1}{\gamma} + \left(\frac{M}{\rho}\right)^\frac{1}{1+\gamma} $$ 
satisfies the inequality
$$ \Psi'(t)+ \rho \Psi(t)^{1+\gamma} \ge M. $$ 
 
Hence $ \Psi(t)\ge\Phi(t)$ for every $t\in(0,T)$ since $\Psi(t) \ge \Phi(t) $ for $t\to 0^+$. \end{proof}

We are now ready to state and prove the universal bound property for  the solutions to (\ref{eqn:abstract}) in the regime $0<\alpha<\beta$. We point out that in our main results below, as well as in Lemma~\ref{lemma:ubp} above, we do not ask solutions to be defined for $t=0$. In some sense, our estimates are universal because they do not depend on initial data, but even better because they do not even require initial data.

\begin{thm}[Universal bound property]\label{thm:ubp}

Let us consider the chain of functional spaces (\ref{7-spaces}), and let $F$ and $g$ be two functions satisfying assumptions (F1)--(F2)--(F3) and (G1)--(G2)--(G3) of section~\ref{sec:functional}. 

Let $T>0$ be a real number, and let $u:(0,T)\to V$ be a strong solution to (\ref{eqn:abstract}) according to Definition~\ref{defn:strongsol}.

Let us assume in addition that $0<\alpha<\beta$, and let us set
\begin{equation}
\gamma:=\min\left\{\frac{\alpha}{2},\frac{\beta -\alpha}{(\alpha+1)(\beta + 2)}\right\}.
\label{defn:gamma}
\end{equation}

Then there exist two real numbers ${\Gamma}$ and $\Gamma_{\!*}$ such that
\begin{equation}
\|u'(t)\|_{H}^{2}+F(u(t))\leq{\Gamma}\,t^{-1/\gamma}+\Gamma_{\!*}
\qquad
\forall t\in(0,T).
\label{th:ubp}
\end{equation}

The constants ${\Gamma}$ and $\Gamma_{\!*}$ depend on the immersions (\ref{7-spaces}), and on the constants that appear in (\ref{hp:F2}) through (\ref{hp:G3}), but they are independent of $T$ and $u$.

\end{thm}

\begin{proof}

To begin with, we introduce the energy
\begin{equation}
{E(t)}:=\frac{1}{2}\|u'(t)\|_{H}^{2}+F(u(t))+C_{1}+1.
\label{defn:Ehat}
\end{equation}

This energy coincides with the energy $E_0(t)$ defined in (\ref{defn:E}) up to an additive constant, and therefore its time-derivative is again given by the right-hand side of (\ref{eqn:E'}). Due to assumption~(\ref{hp:F2}), this new energy is bounded from below by~1, and therefore we can define the modified energy
\begin{equation}
\Phi(t):={E(t)}+\ep{E(t)}^{\gamma}\langle u(t),u'(t)\rangle_{H},
\nonumber
\end{equation}
where $\gamma$ is defined by (\ref{defn:gamma}), and $\ep>0$ is a small parameter. 

We claim that, for every $\ep>0$ small enough, the modified energy $\Phi$ has the following two properties.
\begin{itemize}

\item  It is a small perturbation of $E$ in the sense that
\begin{equation}
%\forall t\in[0,T],\qquad
\frac{1}{2}{E(t)}\leq\Phi(t)\leq\frac{3}{2}{E(t)}
\qquad
\forall t\in(0,T).
\label{th:Ehat-Phi}
\end{equation}

\item  It satisfies the differential inequality 
\begin{equation}
\Phi'(t)\leq-\ep\frac{\delta_{2}}{8}\left(\frac{2}{3}\right)^{\gamma+1}\Phi(t)^{\gamma+1}+\frac{3C_{3}}{2}+2
\label{th:est-Phi'}
\end{equation}
for almost every $t\in(0,T)$, where $\delta_{2}$ and $C_{3}$ are the constants that appear in (\ref{hp:F3}) and (\ref{hp:G2}), respectively.
\end{itemize}

The smallness of $\ep$ depends only on the norms of the continuous imbeddings (\ref{7-spaces}), and on all the constants in (\ref{hp:F2}) through (\ref{hp:G3}), but it does not depend on $T$ and $u$. All the constants $K_{1}$, \ldots, $K_{18}$ that we introduce in the sequel of the proof have the same property.

If we prove the two claims above, then it is enough to select an admissible value of $\ep>0$, and from (\ref{th:Ehat-Phi}), (\ref{th:est-Phi'}), and the conclusion of Lemma~\ref{lemma:ubp} we deduce for some positive constants $\Gamma_1$, $\Gamma_2$
\begin{equation}
{E(t)}\leq 2\Phi(t)\leq 2\Gamma_{1}\,t^{-1/\gamma}+2\Gamma_{2},
\nonumber
\end{equation}
which implies (\ref{th:ubp}).

\subparagraph{\textmd{\textit{Equivalence of the energies}}}

We show that (\ref{th:Ehat-Phi}) holds true whenever $\ep>0$ is small enough. To begin with, we observe that, when $\alpha$ varies in $(0,\beta)$, the value $\gamma$ defined by (\ref{defn:gamma}) is the minimum between an increasing and a decreasing function of $\alpha$. The two functions coincide when $\alpha=\beta/(\beta+2)$, and in this case $\gamma$ attains its maximum. This proves that $\gamma\leq\beta/(2\beta+4)$, and in particular
\begin{equation}
\gamma+\frac{1}{\beta+2}+\frac{1}{2}\leq 1.
\label{ineq:gamma-beta}
\end{equation}

Now from (\ref{defn:Ehat}) we deduce that
\begin{equation}
\|u'(t)\|_{H}\leq\left[2{E(t)}\right]^{1/2},
\nonumber
\end{equation}
while from the continuous imbedding $Y\subseteq H$ and (\ref{hp:F2}) we deduce that
\begin{equation}
\|u(t)\|_{H}\leq K_{1}\|u(t)\|_{Y}\leq K_{2}\left(F(u(t))+C_{1}\right)^{1/(\beta+2)}\leq K_{2}{E(t)}^{1/(\beta+2)},
\nonumber
\end{equation}
and therefore
\begin{equation}
{E(t)}^{\gamma}\cdot|\langle u(t),u'(t)\rangle_{H}|\leq{E(t)}^{\gamma}\cdot\|u(t)\|_{H}\cdot\|u'(t)\|_{H}\leq K_{3}{E(t)}^{\gamma+(1/2)+1/(\beta+2)}.
\nonumber
\end{equation}

Recalling that ${E(t)}\geq 1$, and keeping (\ref{ineq:gamma-beta}) into account, we conclude that
\begin{equation}
{E(t)}^{\gamma}\cdot|\langle u(t),u'(t)\rangle_{H}|\leq K_{3}{E(t)}
\label{est:Phi-interaction}
\end{equation}
so that (\ref{th:Ehat-Phi}) holds true whenever $K_{3}\ep\leq 1/2$.

\subparagraph{\textmd{\textit{Differential inequality for the modified energy}}}

We show that $\Phi$ satisfies (\ref{th:est-Phi'}) when $\ep>0$ is small enough. To begin with, we compute the time-derivative of $\Phi$, and we write it in the form
\begin{eqnarray}
\Phi'(t) & = &  -\langle g(t,u'(t)),u'(t)\rangle_{V',V}\left(1+\gamma\ep{E(t)}^{\gamma-1}\langle u(t),u'(t)\rangle_{H}\right) 
\nonumber  \\[0.5ex]
  &  &  \mbox{}+\ep{E(t)}^{\gamma}\cdot\left(\|u'(t)\|_{H}^{2}-\langle\nabla F(u(t)),u(t)\rangle_{V',V}\right)   
\nonumber  \\[1ex]
  &  &  \mbox{}-\ep{E(t)}^{\gamma}\langle g(t,u'(t)),u(t)\rangle_{V',V}. 
\label{est:Phi'}
\end{eqnarray}

Let $L_{1}$, $L_{2}$, and $L_{3}$ denote the terms in the three lines of the right-hand side. 
\begin{itemize}

\item  Let us estimate $L_{1}$. From (\ref{est:Phi-interaction}) we deduce that
\begin{equation}
{E(t)}^{\gamma-1}\left|\langle u(t),u'(t)\rangle_{H}\strut\right|\leq K_{3}
\nonumber
\end{equation}
and therefore
\begin{equation}
\frac{1}{2}\leq 1+\gamma\ep E(t)^{\gamma-1}\langle u(t),u'(t)\rangle_{H}\leq\frac{3}{2}
\label{est:L1-T2}
\end{equation}
provided that $\ep$ is small enough. Now we distinguish two cases.
\begin{itemize}

\item  If $\langle g(t,u'(t)),u'(t)\rangle_{V',V}\leq 0$, then from (\ref{est:L1-T2}) we obtain that
\begin{equation}
L_{1}\leq -\frac{3}{2}\langle g(t,u'(t)),u'(t)\rangle_{V',V},
\nonumber
\end{equation}
and hence from (\ref{hp:G2}) we conclude that
\begin{equation}
L_{1}\leq\frac{3C_{3}}{2}-\frac{3\delta_{3}}{2}\|u'(t)\|_{X}^{\alpha+2}.
\nonumber
\end{equation}

\item  If $\langle g(t,u'(t)),u'(t)\rangle_{V',V}\geq 0$, then from (\ref{est:L1-T2}) we obtain that
\begin{equation}
L_{1}\leq -\frac{1}{2}\langle g(t,u'(t)),u'(t)\rangle_{V',V},
\nonumber
\end{equation}
and hence from (\ref{hp:G2}) we conclude that
\begin{equation}
L_{1}\leq\frac{C_{3}}{2}-\frac{\delta_{3}}{2}\|u'(t)\|_{X}^{\alpha+2}.
\nonumber
\end{equation}

\end{itemize}

In both cases it is true that
\begin{equation}
L_{1}\leq\frac{3C_{3}}{2}-\frac{\delta_{3}}{2}\|u'(t)\|_{X}^{\alpha+2}.
\nonumber
\end{equation}

\item Let us estimate $L_{2}$. From (\ref{hp:F3}) and (\ref{defn:Ehat}) it follows that
\begin{equation}
\langle\nabla F(u(t)),u(t)\rangle_{V',V}\geq\delta_{2}F(u(t))-C_{2}=\delta_{2}{E(t)}-K_{4}-K_{5}\|u'(t)\|_{H}^{2},
\nonumber
\end{equation}
and hence, due to the continuous imbedding $X\subseteq H$, we obtain that
\begin{eqnarray}
\frac{1}{\ep}L_{2} & \leq & -\delta_{2}{E(t)}^{\gamma+1}+K_{4}{E(t)}^{\gamma}+K_{6}{E(t)}^{\gamma}\|u'(t)\|_{H}^{2} \nonumber \\
& \leq & -\delta_{2}{E(t)}^{\gamma+1}+K_{4}{E(t)}^{\gamma}+K_{7}{E(t)}^{\gamma}\|u'(t)\|_{X}^{2}
\label{est:L2-main}
\end{eqnarray}

The first condition in the definition (\ref{defn:gamma}) of $\gamma$ implies that
\begin{equation}
\gamma\leq\frac{\alpha(\gamma+1)}{\alpha+2},
\label{est:L2-gamma}
\end{equation}
and therefore, since ${E(t)}\geq 1$, we find that
\begin{equation}
{E(t)}^{\gamma}\leq K_{8}{E(t)}^{\alpha(\gamma+1)/(\alpha+2)},
\label{est:L2-step}
\end{equation}
actually in this case with $K_{8}=1$. Applying this inequality to the last term of (\ref{est:L2-main}) we obtain that
\begin{equation}
\frac{1}{\ep}L_{2}\leq -\delta_{2}{E(t)}^{\gamma+1}+K_{4}{E(t)}^{\gamma}+K_{9}{E(t)}^{\alpha(\gamma+1)/(\alpha+2)}\|u'(t)\|_{X}^{2}.
\label{est:L2-main-bis}
\end{equation}

In order to estimate the second term in (\ref{est:L2-main-bis}) we observe that
\begin{equation}
K_{4}{E(t)}^{\gamma}\leq\frac{\delta_{2}}{4}{E(t)}^{\gamma+1}+K_{10}.
\label{est:L2-1}
\end{equation}

In order to estimate the third term in (\ref{est:L2-main-bis}), we apply Young's inequality and we deduce that
\begin{equation}
K_{9}{E(t)}^{\alpha(\gamma+1)/(\alpha+2)}\cdot\|u'(t)\|_{X}^{2}\leq\frac{\delta_{2}}{2}{E(t)}^{\gamma+1}+K_{11}\|u'(t)\|_{X}^{\alpha+2}.
\label{est:L2-2}
\end{equation}

Plugging (\ref{est:L2-1}) and (\ref{est:L2-2}) into (\ref{est:L2-main-bis}) we conclude that
\begin{equation}
\frac{1}{\ep}L_{2} \leq -\frac{\delta_{2}}{4}{E(t)}^{\gamma+1}+K_{11}\|u'(t)\|_{X}^{\alpha+2}+K_{10},
\nonumber
\end{equation}
which means  that when $\ep>0$ is small enough we obtain that
\begin{equation}
L_{2}\leq-\ep\frac{\delta_{2}}{4}{E(t)}^{\gamma+1}+\frac{\delta_{3}}{4}\|u'(t)\|_{X}^{\alpha+2}+1.
\nonumber
\end{equation} 

\item Let us estimate $L_{3}$. From (\ref{hp:G3}) it follows that
\begin{eqnarray}
\left|\langle g(t,u'(t)),u(t)\rangle_{V',V}\strut\right|  &  \leq  &  \|g(t,u'(t))\|_{X'}\|u(t)\|_{X}  
\nonumber  \\[1ex]
 & \leq & \left(C_{4}+{D_4}\|u'(t)\|_{X}^{\alpha+1}\right)\|u(t)\|_{X}.
 \label{est:L3-start}
\end{eqnarray}

From the continuous imbedding $Y\subseteq X$, assumption~(\ref{hp:F2}), and definition~(\ref{defn:Ehat}), we obtain that
\begin{equation}
\|u(t)\|_{X}\leq K_{12}\|u(t)\|_{Y}\leq K_{13}\left[F(u(t))+C_{1}\right]^{1/(\beta+2)}\leq K_{13}{E(t)}^{1/(\beta+2)}.
\label{est:L3-uX}
\end{equation}

The second condition in (\ref{defn:gamma}) implies that
\begin{equation}
\gamma+\frac{1}{\beta+2}\leq\frac{\gamma+1}{\alpha+2},
\label{est:L3-beta}
\end{equation}
and therefore, since ${E(t)}\geq 1$, from (\ref{est:L3-uX}) we obtain that
\begin{equation}
{E(t)}^{\gamma}\cdot\|u(t)\|_{X}\leq K_{13}{E(t)}^{\gamma+1/(\beta+2)}\leq K_{14}{E(t)}^{(\gamma+1)/(\alpha+2)}.
\label{est:L3-step}
\end{equation}

From this inequality and (\ref{est:L3-start}) we deduce that
\begin{equation}
\frac{1}{\ep}L_{3}\leq K_{15}{E(t)}^{(\gamma+1)/(\alpha+2)}+K_{16}{E(t)}^{(\gamma+1)/(\alpha+2)}\|u'(t)\|_{X}^{\alpha+1}.
\label{est:L4-main}
\end{equation}

In order to estimate the first term in (\ref{est:L4-main}), we observe that $(\gamma+1)/(\alpha+2)<\gamma+1$, and therefore
\begin{equation}
K_{15}{E(t)}^{(\gamma+1)/(\alpha+2)}\leq\frac{\delta_{2}}{16}{E(t)}^{\gamma+1}+K_{17}.
\label{est:L4-1}
\end{equation}

In order to estimate the second term in (\ref{est:L4-main}), we apply Young's inequality and we obtain that
\begin{equation}
K_{16}{E(t)}^{(\gamma+1)/(\alpha+2)}\|u'(t)\|_{X}^{\alpha+1}\leq\frac{\delta_{2}}{16}{E(t)}^{\gamma+1}+K_{18}\|u'(t)\|_{X}^{\alpha+2}.
\label{est:L4-2}
\end{equation}

Plugging (\ref{est:L4-1}) and (\ref{est:L4-2}) into (\ref{est:L4-main}) we conclude that
\begin{equation}
\frac{1}{\ep}L_{3}\leq\frac{\delta_{2}}{8}{E(t)}^{\gamma+1}+K_{18}\|u'(t)\|_{X}^{\alpha+2}+K_{17},
\nonumber
\end{equation}
which means  that when $\ep>0$ is small enough we obtain that
\begin{equation}
L_{3}\leq\ep\frac{\delta_{2}}{8}{E(t)}^{\gamma+1}+\frac{\delta_{3}}{4}\|u'(t)\|_{X}^{\alpha+2}+1.
\nonumber
\end{equation}

\end{itemize}

From the estimates for $L_{1}$, $L_{2}$, and $L_{3}$ we conclude that
\begin{equation}
\Phi'(t)\leq-\ep\frac{\delta_{2}}{8}{E(t)}^{\gamma+1}+\frac{3C_{3}}{2}+2
\nonumber
\end{equation}
whenever $\ep>0$ is small enough, and this implies (\ref{th:est-Phi'}) because of (\ref{th:Ehat-Phi}).
\end{proof}

%\clearpage

\subsection{Universal decay at infinity} 

In this section we prove two universal decay properties for solutions to (\ref{eqn:abstract}). In the first result we strengthen the assumptions by requiring that (F2)--(F3) and (G2)--(G3) hold true with all $C_{i}$'s equal to zero. We obtain that the energy of solutions is bounded from above by a multiple of a negative power of $t$, independent (both the power and the constant) of the solution itself.

At a first glance, the conclusion (\ref{th:ud1}) resembles (\ref{th:ubp}) with $\Gamma_{\!*}=0$. Nevertheless, we stress that the value of $\gamma$ is now different (it is the maximum instead of the minimum between the same two quantities), and that now the conclusion is true for every $t\geq 1$ for solutions that are defined on the whole half-line $(0,+\infty)$.

\begin{thm}[Universal decay under standard assumptions]\label{thm:ud1}

Let us consider the chain of functional spaces (\ref{7-spaces}), and let $F$ and $g$ be two functions satisfying assumptions (F1)--(F2)--(F3) and (G1)--(G2)--(G3) of section~\ref{sec:functional} with $C_{1}=C_{2}=C_{3}=C_{4}=0$ and $T=+\infty$.  Let $u:(0,+\infty)\to V$ be a strong solution to (\ref{eqn:abstract}) according to Definition~\ref{defn:strongsol}.

Let us assume in addition that $0<\alpha<\beta$, and let us set
\begin{equation}
\gamma:=\max\left\{\frac{\alpha}{2},\frac{\beta -\alpha}{(\alpha+1)(\beta + 2)}\right\}.
\label{defn:gamma-as}
\end{equation}

Then there exists a real number ${D}$ such that
\begin{equation}
\|u'(t)\|_{H}^{2}+F(u(t))\leq{D}\,t^{-1/\gamma}
\qquad
\forall t\geq 1.
\label{th:ud1}
\end{equation}

The constant ${D}$ depends on the immersions (\ref{7-spaces}), and on the constants that appear in (\ref{hp:F2}) through (\ref{hp:G3}), but it is independent of $u$.

\end{thm}

\begin{proof}

We consider the usual energy $E_{0}(t)$ defined in (\ref{defn:E}). From Theorem~\ref{thm:ubp} we know that the universal bound
\begin{equation}
E_{0}(t)\leq K_{1}
\qquad
\forall t\geq 1
\label{th:E(1)}
\end{equation}
holds true with a constant independent of the solution. Then we introduce the modified energy
\begin{equation}
\Phi(t):=E_{0}(t)+\ep E_{0}(t)^{\gamma}\langle u(t),u'(t)\rangle_{H},
\label{defn:E-ud1}
\end{equation}
where $\gamma$ is defined by (\ref{defn:gamma-as}), and $\ep>0$ is a small parameter. Since $\gamma$ might be less than~1, in order to avoid differentiability issues we assume for the time being that $E_{0}(t)>0$ for every $t>0$. At the end of the proof we discuss the case where $E_{0}(t)=0$ for some $t>0$.

We claim that, for every $\ep>0$ small enough, the modified energy $\Phi$ has the following two properties.
\begin{itemize}

\item  It is a small perturbation of $E_{0}$ in the sense that
\begin{equation}
\frac{1}{2}E_{0}(t)\leq\Phi(t)\leq\frac{3}{2}E_{0}(t).
\label{th:E-Phi}
\end{equation}

\item  It satisfies the differential inequality 
\begin{equation}
\Phi'(t)\leq-\ep\frac{\delta_{2}}{4}\left(\frac{2}{3}\right)^{\gamma+1} \Phi(t)^{\gamma+1}
\label{th:est-Phi'-as}
\end{equation}
for almost every $t\geq 1$, where $\gamma$ is defined by (\ref{defn:gamma-as}), and $\delta_{2}$ is the constant that appears in (\ref{hp:F3}).
\end{itemize}

As usual, the smallness of $\ep$ depends only on the norms of the continuous imbeddings (\ref{7-spaces}), and on all the constants appearing in (\ref{hp:F2}) through (\ref{hp:G3}), but it does not depend on $u$. All the constants $K_{2}$, \ldots, $K_{10}$ that we introduce in the sequel of the proof have the same property.

Let us assume that the two claims above have been proved. From the universal bound~(\ref{th:E(1)}) we know that $\Phi(1)\leq 3K_{1}/2$. At this point it is enough to select an admissible value of $\ep>0$, and integrating the differential inequality (\ref{th:est-Phi'-as}) we deduce that
\begin{equation}
\Phi(t)\leq K_{2}\,t^{-1/\gamma}
\qquad
\forall t\geq 1,
\nonumber
\end{equation}
which implies (\ref{th:ud1}) because of (\ref{th:E-Phi}).

\subparagraph{\textmd{\textit{Equivalence of the energies}}}

Arguing as in the corresponding paragraph of the proof of Theorem~\ref{thm:ubp} we find that now
\begin{equation}
\gamma+\frac{1}{\beta+2}+\frac{1}{2}\geq 1,
\label{ineq:gamma-beta-as}
\end{equation}
and again
\begin{equation}
E_{0}(t)^{\gamma}\,|\langle u(t),u'(t)\rangle_{H}|\leq K_{3}\left[E_{0}(t)\right]^{\gamma+(1/2)+1/(\beta+2)}
\nonumber
\end{equation}

Recalling (\ref{ineq:gamma-beta-as}) and the universal bound (\ref{th:E(1)}) we conclude that
\begin{equation}
E_{0}(t)^{\gamma}\,|\langle u(t),u'(t)\rangle_{H}|\leq K_{4}E_{0}(t).
\label{est:Phi-interaction-as}
\end{equation}
and hence (\ref{th:E-Phi}) holds true whenever $K_{4}\ep\leq 1/2$.

\subparagraph{\textmd{\textit{Differential inequality for the modified energy}}}

We show that $\Phi$ satisfies (\ref{th:est-Phi'-as}) when $\ep>0$ is small enough. The time-derivative of $\Phi$ is given by (\ref{est:Phi'}), now with $E_{0}(t)$ instead of ${E(t)}$. Let $L_{1}$, $L_{2}$, and $L_{3}$ denote the terms in the three lines of the right-hand side. 
\begin{itemize}

\item  Let us estimate $L_{1}$. From (\ref{est:Phi-interaction-as}) we deduce that
\begin{equation}
1+\gamma\ep E_{0}(t)^{\gamma-1} \langle u(t),u'(t)\rangle_{H}\geq\frac{1}{2}
\nonumber
\end{equation}
provided that $\ep>0$ is small enough. Since now (\ref{hp:G2}) holds true with $C_{3}=0$, we conclude that
\begin{equation}
L_{1}\leq-\frac{\delta_{3}}{2}\|u'(t)\|_{X}^{\alpha+2}.
\nonumber
\end{equation}

\item Let us estimate $L_{2}$. Arguing as in the corresponding paragraph of the proof of Theorem~\ref{thm:ubp}, and recalling that now $C_{2}=0$, we find that
\begin{equation}
\frac{1}{\ep}L_{2}\leq -\delta_{2}E_{0}(t)^{\gamma+1}+K_{5}E_{0}(t)^{\gamma}\|u'(t)\|_{X}^{2}.
\nonumber
\end{equation}

The new definition (\ref{defn:gamma-as}) of $\gamma$ implies that (\ref{est:L2-gamma}) holds true with the opposite sign, but on the other hand now we know that $E_{0}(t)$ is bounded from above for $t\geq 1$, and therefore again the equivalent of inequality (\ref{est:L2-step}) holds true, and therefore
\begin{equation}
\frac{1}{\ep}L_{2}\leq -\delta_{2}E_{0}(t)^{\gamma+1}+K_{6}E_{0}(t)^{\alpha(\gamma+1)/(\alpha+2)}\|u'(t)\|_{X}^{2}.
\nonumber
\end{equation}

We estimate the last term by exploiting Young's inequality as we did in (\ref{est:L2-2}), and  we obtain that
\begin{equation}
\frac{1}{\ep}L_{2} \leq -\frac{\delta_{2}}{2}E_{0}(t)^{\gamma+1}+K_{7}\|u'(t)\|_{X}^{\alpha+2},
\nonumber
\end{equation}
which means  that when $\ep>0$ is small enough we conclude that
\begin{equation}
L_{2}\leq-\ep\frac{\delta_{2}}{2}E_{0}(t)^{\gamma+1}+\frac{\delta_{3}}{4}\|u'(t)\|_{X}^{\alpha+2}.
\nonumber
\end{equation} 

\item Let us estimate $L_{3}$. Arguing as in the corresponding paragraph of the proof of Theorem~\ref{thm:ubp}, and recalling that now $C_{4}=0$, we find that
\begin{equation}
\frac{1}{\ep}L_{3}\leq E_{0}(t)^{\gamma}\cdot{D_4}\|u'(t)\|_{X}^{\alpha+1}\cdot\|u(t)\|_{X}\leq K_{8}E_{0}(t)^{\gamma+1/(\beta+2)}\|u'(t)\|_{X}^{\alpha+1}.
\nonumber
\end{equation}

The new definition (\ref{defn:gamma-as}) of $\gamma$ implies that (\ref{est:L3-beta}) holds true with the opposite sign, but on the other hand now we know that $E_{0}(t)$ is bounded from above for $t\geq 1$, and therefore again inequality (\ref{est:L3-step}) holds true. It follows that 
\begin{equation}
\frac{1}{\ep}L_{3}\leq K_{9}E_{0}(t)^{(\gamma+1)/(\alpha+2)}\|u'(t)\|_{X}^{\alpha+1}.
\nonumber
\end{equation}

We estimate the right-hand side by exploiting Young's inequality, and we find that
\begin{equation}
\frac{1}{\ep}L_{3}\leq\frac{\delta_{2}}{4}E_{0}(t)^{\gamma+1}+K_{10}\|u'(t)\|_{X}^{\alpha+2},
\nonumber
\end{equation}
which means  that when $\ep>0$ is small enough we obtain that
\begin{equation}
L_{3}\leq\ep\frac{\delta_{2}}{4}E_{0}(t)^{\gamma+1}+\frac{\delta_{3}}{4}\|u'(t)\|_{X}^{\alpha+2}.
\nonumber
\end{equation}

\end{itemize}

Plugging the estimates for $L_{1}$, $L_{2}$, $L_{3}$ into the expression of $\Phi'(t)$, we conclude that
\begin{equation}
\Phi'(t)\leq-\frac{\delta_{2}}{4}\ep E_{0}(t)^{\gamma+1}
\nonumber
\end{equation}
for almost every $t\geq 1$, and this implies (\ref{th:est-Phi'-as}) because of (\ref{th:E-Phi}).

\subparagraph{\textmd{\textit{When the energy vanishes for some positive time}}}

It remains to consider the case where $E_{0}(t)$ vanishes for some positive time. To begin with, from (\ref{eqn:E'}) and assumption (\ref{hp:G2}) with $C_{3}=0$ we deduce that $E_{0}(t)$ is nonincreasing. It follows that there are only two cases. If $E_{0}(t)=0$ for every $t\geq 1$, then the conclusion is trivial. Otherwise, there exists $T>1$ such that $E_{0}(t)>0$ in $[1,T)$, and $E_{0}(t)=0$ for every $t\geq T$. In this case the conclusion is trivial for $t\geq T$, while in the interval $[1,T)$ both (\ref{th:E-Phi}) and (\ref{th:est-Phi'-as}) hold true, leading to (\ref{th:ud1}) also in this interval.
\end{proof}

%\clearpage

In the last result we assume in addition that the chain of spaces (\ref{7-spaces}) satisfies (\ref{hp:norms}), and that the function $F$ satisfies also~(F4). We obtain that the universal decay holds true with a better exponent, equal to the first term in the maximum~(\ref{defn:gamma-as}).

\begin{thm}[Universal decay under stronger assumptions]\label{thm:ud2}

Let us consider the chain of functional spaces (\ref{7-spaces}), and let us assume that (\ref{hp:norms}) holds true. Let $F$ and $g$ be two functions satisfying assumptions (F1)--(F2)--(F3)--(F4) and (G1)--(G2)--(G3) of section~\ref{sec:functional} with $C_{1}=C_{2}=C_{3}=C_{4}=0$ and $T=+\infty$.  Let $u:(0,+\infty)\to V$ be a strong solution to (\ref{eqn:abstract}) according to Definition~\ref{defn:strongsol}.

Let us assume in addition that $0<\alpha<\beta$.

Then there exists a real number ${D}$ such that
\begin{equation}
\|u'(t)\|_{H}^{2}+F(u(t))\leq{D}\,t^{-2/\alpha}
\qquad
\forall t\geq 1.
\nonumber
\end{equation}

The constant ${D}$ depends on the immersions (\ref{7-spaces}), and on the constants that appear in all the assumptions, but it is independent of $u$.

\end{thm}

\begin{proof}

We consider the usual energy $E_{0}(t)$ defined in~(\ref{defn:E}), and the modified energy $\Phi(t)$ defined in~(\ref{defn:E-ud1}), where now $\gamma:=\alpha/2$, and $\ep>0$ is again a small parameter. We assume, without loss of generality, that $E_{0}(t)>0$ for every positive time, because otherwise we can argue as in the last paragraph of the previous proof.  Again we claim that, for every $\ep>0$ small enough, the modified energy $\Phi$ is a small perturbation of $E_{0}$ in the sense of (\ref{th:E-Phi}), and it satisfies the differential inequality (\ref{th:est-Phi'-as}) for almost every $t\geq 1$. All the constants and the smallness of $\ep$ depend as usual on the constants appearing in the assumptions, but they do not depend on $u$. 

Again we know from Theorem~\ref{thm:ubp} that the universal bound (\ref{th:E(1)})
holds true, and therefore the conclusion follows again from the differential inequality and the equivalence of the energies.

\subparagraph{\textmd{\textit{Equivalence of the energies}}}

From assumption (\ref{hp:F4}) we obtain that
\begin{equation}
|\langle u'(t),u(t)\rangle_{H}|\leq\frac{1}{2}\|u'(t)\|_{H}^{2}+\frac{1}{2}\|u(t)\|_{H}^{2}\leq K_{1}E_{0}(u(t)),
\nonumber
\end{equation}
and hence, since $E_{0}(t)$ is bounded from above for $t\geq 1$, we conclude that
\begin{equation}
\ep E_{0}(t)^{\gamma}|\langle u(t),u'(t)\rangle_{H}|\leq\ep K_{1}E_{0}(t)^{\gamma}\cdot E_{0}(t)\leq\ep K_{2}E_{0}(t).
\nonumber
\end{equation}

This proves that (\ref{th:E-Phi}) holds true whenever $K_{2}\ep\leq 1/2$.

\subparagraph{\textmd{\textit{Differential inequality for the modified energy}}}

We show that $\Phi$ satisfies (\ref{th:est-Phi'-as}) when $\ep>0$ is small enough. The time-derivative of $\Phi$ is again given by (\ref{est:Phi'}) with $E_{0}$ instead of $E$. Let $L_{1}$, $L_{2}$, and $L_{3}$ denote the terms in the three lines of the right-hand side. 

We can estimate $L_{1}$ and $L_{2}$ by arguing as in the corresponding parts of proof of Theorem~\ref{thm:ud1}, because in those parts we exploited only that $\gamma\geq\alpha/2$. When $\ep>0$ is small enough we obtain that
\begin{equation}
L_{1}\leq-\frac{\delta_{3}}{2}\|u'(t)\|_{X}^{\alpha+2}
\qquad\mbox{and}\qquad
L_{2}\leq-\ep\frac{\delta_{2}}{2}{E_{0}(t)}^{\gamma+1}+\frac{\delta_{3}}{4}\|u'(t)\|_{X}^{\alpha+2}.
\nonumber
\end{equation}

In order to estimate $L_{3}$, from (\ref{hp:G3}) with $C_{4}=0$ we deduce that
\begin{equation}
\frac{1}{\ep}L_{3}\leq{D_4}E_{0}(t)^{\gamma}\cdot\|u'(t)\|_{X}^{\alpha+1}\cdot\|u(t)\|_{X}.
\nonumber
\end{equation}

From (\ref{hp:norms}), (\ref{hp:F2}) and (\ref{hp:F4}) we obtain that
\begin{equation}
\|u\|_{X}^{\alpha+2}\leq{C_5}\left(\|u\|_{H}^{2}+\|u\|_{Y}^{\beta+2}\right)\leq K_{3}F(u(t))\leq K_{3}E_{0}(t),
\nonumber
\end{equation}
and therefore
\begin{equation}
\frac{1}{\ep}L_{3}\leq K_{4}E_{0}(t)^{\gamma+1/(\alpha+2)}\cdot\|u'(t)\|_{X}^{\alpha+1}.
\nonumber
\end{equation}

Now we observe that
\begin{equation}
\gamma+\frac{1}{\alpha+2}\geq\frac{\gamma+1}{\alpha+2},
\nonumber
\end{equation}
and thus from the universal bound (\ref{th:E(1)}) we deduce that
\begin{equation}
\frac{1}{\ep}L_{3}\leq K_{5}E_{0}(t)^{(\gamma+1)/(\alpha+2)}\cdot\|u'(t)\|_{X}^{\alpha+1}.
\nonumber
\end{equation}

Finally, from Young's inequality we conclude that
\begin{equation}
\frac{1}{\ep}L_{3}\leq\frac{\delta_{2}}{4}{E_{0}(t)}^{\gamma+1}+K_{6}\|u'(t)\|_{X}^{\alpha+2},
\nonumber
\end{equation}
which means  that when $\ep>0$ is small enough we obtain that
\begin{equation}
L_{3}\leq\ep\frac{\delta_{2}}{4}{E_{0}(t)}^{\gamma+1}+\frac{\delta_{3}}{4}\|u'(t)\|_{X}^{\alpha+2}.
\nonumber
\end{equation}

At this point the conclusion follows as in the proof of Theorem~\ref{thm:ud1}.
\end{proof}

%\clearpage

\section{Universal bound/decay for PDEs}\label{sec:pde}

In this section we apply the abstract result of section~\ref{sec:abstract} to some hyperbolic partial differential equations. Throughout this section, we assume that $N$ is a positive integer, $\Omega\subseteq\re^{N}$ is a bounded open set with smooth boundary (regular enough to have classical Sobolev imbeddings and $H^{2}$ regularity for the Dirichlet or Neumann Laplacian up to the boundary), $T$ is a positive real number, and $h\in L^{\infty}((0,T),L^{2}(\Omega))$ is a  function that plays the role of a forcing term in the equations.

\subsection{Semilinear wave equations}%\label{sec:wave}

\paragraph{\textmd{\textit{Statement of the problem and well-posedness}}}

Let us consider, in a cylinder $(0,T)\times\Omega$ or $(0,+\infty)\times\Omega$, semilinear wave equations of the form
\begin{equation}
u_{tt}-\Delta u+b|u|^{\beta}u-\lambda u+c|u_{t}|^{\alpha}u_{t}-\mu u_{t}=h,
\label{eqn:wave}
\end{equation}
where $\alpha$, $\beta$, $b$, $c$ are positive real parameters, and $\lambda$, $\mu$ are real parameters. Let us add initial conditions
\begin{equation}
u(0,x)=u_{0}(x),
\qquad
u_{t}(0,x)=u_{1}(x),
\label{initial-data}
\end{equation}
and either homogeneous Dirichlet boundary conditions
\begin{equation}
u(t,x)=0
\qquad
\mbox{in }(0,T)\times\partial\Omega,
\nonumber
\end{equation}
or homogeneous Neumann boundary conditions
\begin{equation}
\frac{\partial u}{\partial n}(t,x)=0
\qquad
\mbox{in }(0,T)\times\partial\Omega.
\nonumber
\end{equation}

If we assume that 
\begin{equation}
(N-2)\beta\leq 2,
\label{hp:beta-wave}
\end{equation}
then $H^{1}(\Omega)$ is continuously imbedded into $L^{2\beta+2}(\Omega)$. In this case the initial-boundary-value problem is classically well-posed for initial data $(u_{0},u_{1})\in H^{1}_{0}(\Omega)\times L^{2}(\Omega)$ in the case of Dirichlet boundary conditions, and for initial data $(u_{0},u_{1})\in H^{1}(\Omega)\times L^{2}(\Omega)$ in the case of Neumann boundary conditions. Here ``well-posedness'' refers to weak solutions, while strong solutions exist under additional regularity assumptions on the initial data and on the forcing term $h$. We refer to Proposition~II.2.2.1 and Theorem~II.3.2.1 in~\cite{HDieudo} for the construction of weak and strong solutions.

\paragraph{\textmd{\textit{The abstract framework}}}

The problem fits in the abstract framework of section~\ref{sec:functional} if we set
\begin{equation}
F(u):=\frac{1}{2}\|\nabla u\|_{L^{2}(\Omega)}^{2}-\frac{\lambda}{2}\|u\|_{L^{2}(\Omega)}^{2}+\frac{b}{\beta+2}\|u\|_{L^{\beta+2}(\Omega)}^{\beta+2},
\nonumber
\end{equation}
and
\begin{equation}
[g(t,v)](t,x):=c|v(x)|^{\alpha}v(x)-\mu v(x)-h(t,x),
\label{defn:gtv-wave}
\end{equation}
and we choose the functional spaces
\begin{equation}
H:=L^{2}(\Omega),
\qquad
X:=L^{\alpha+2}(\Omega),
\qquad
Y:=L^{\beta+2}(\Omega),
\label{defn:HXY-wave}
\end{equation}
with $V:=H^{1}_{0}(\Omega)$ in the case of Dirichlet boundary conditions, and $V:=H^{1}(\Omega)$ in the case of Neumann boundary conditions. The verification of (F1) and (G1) is quite straightforward with this choice of the functional spaces. Let us check the remaining abstract assumptions of section~\ref{sec:functional}.

\begin{itemize}

\item  Assumption (F2). Inequality (\ref{hp:F2}) holds true, both in the Neumann and in the Dirichlet case, because of the super-quadratic power $\beta+2$. If we want (\ref{hp:F2}) to be true with $C_{1}=0$, then in the Neumann case we have to assume that $\lambda\leq 0$, while in the Dirichlet case it is enough to assume that $\lambda\leq\lambda_{1}(\Omega)$, where $\lambda_{1}(\Omega)$ denotes the first eigenvalue of the Dirichlet Laplacian in $\Omega$.

\item  Assumption (F3). Inequality (\ref{hp:F3}) is always true with $\delta_{2}=2$ and $C_{2}=0$, for every admissible value of the parameters.

\item  Assumption (G2). From (\ref{defn:gtv-wave}) it follows that
\begin{eqnarray*}
\langle g(t,v),v\rangle_{V',V} & = & c\|v\|_{L^{\alpha+2}(\Omega)}^{\alpha+2}-\mu\|v\|_{L^{2}(\Omega)}^{2}-\int_{\Omega}h(t,x)v(x)\,dx \\[1ex]
& \geq & c\|v\|_{L^{\alpha+2}(\Omega)}^{\alpha+2}-\mu\|v\|_{L^{2}(\Omega)}^{2}-\|h(t,x)\|_{L^{2}(\Omega)}\|v\|_{L^{2}(\Omega)}.
\end{eqnarray*}

At this point inequality (\ref{hp:G2}) follows from the imbedding $L^{\alpha+2}(\Omega)\subseteq L^{2}(\Omega)$, and from the super-quadratic growth of the power $\alpha+2$.

If we want (\ref{hp:G2}) to hold true with $C_{3}=0$, we have to assume both that $\mu\leq 0$, and that $h\equiv 0$.

\item  Assumption (G3). Setting for simplicity $\sigma:=(\alpha+2)/(\alpha+1)$, we observe that $X'=L^{\sigma}(\Omega)$. From (\ref{defn:gtv-wave}) we deduce that
\begin{eqnarray}
\left|[g(t,v)](t,x)\right| & \leq & c|v(x)|^{\alpha+1}+|\mu|\cdot |v(x)|+|h(t,x)| 
\nonumber  \\[0.5ex]
& \leq & K_{1}|v(x)|^{\alpha+1}+K_{2}+|h(t,x)|.
\label{est:G2-2-wave}
\end{eqnarray}

Since $\sigma\leq 2$, when we compute the norm in $L^{\sigma}(\Omega)$ we obtain that 
\begin{equation}
\left\|K_{2}+|h(t,x)|\,\strut\right\|_{L^{\sigma}(\Omega)}\leq K_{3}+\|h\|_{L^{\sigma}(\Omega)}\leq K_{3}+K_{4}\|h\|_{L^{2}(\Omega)}\leq K_{5},
\nonumber
\end{equation}
and
\begin{equation}
\left\|K_{1}|v(x)|^{\alpha+1}\right\|_{L^{\sigma}(\Omega)}=K_{1}\|v\|_{L^{\alpha+2}(\Omega)}^{\alpha+1}.
\nonumber
\end{equation}

Plugging these two estimates into (\ref{est:G2-2-wave}) we obtain (\ref{hp:G3}).

Finally, (\ref{hp:G3}) holds true with $C_{4}=0$ of we assume that $\mu=0$ and $h\equiv 0$.

\item  Assumption (F4). Inequality (\ref{hp:F4}) holds true in the case of Dirichlet boundary conditions if $\lambda<\lambda_{1}(\Omega)$. Indeed, in this case we can apply Poincaré inequality and deduce that
\begin{equation}
F(u)\geq\frac{1}{2}\|\nabla u\|_{L^{2}(\Omega)}^{2}-\frac{\lambda}{2}\|u\|_{L^{2}(\Omega)}^{2}\geq\frac{1}{2}(\lambda_{1}(\Omega)-\lambda)\|u\|_{L^{2}(\Omega)}^{2}.
\nonumber
\end{equation}

In the case of Neumann boundary conditions, inequality (\ref{hp:F4}) holds true if $\lambda<0$ (and indeed 0 is the first eigenvalue of the Neumann Laplacian).

\item  Assumption (\ref{hp:norms}). In the regime $0<\alpha<\beta$ this inequality is always true with ${C_5}=1$ because it amounts to saying that
\begin{equation}
\int_{\Omega}|u(x)|^{\alpha+2}\,dx\leq\int_{\Omega}|u(x)|^{2}\,dx+\int_{\Omega}|u(x)|^{\beta+2}\,dx,
\nonumber
\end{equation}
which in turn follows from the inequality $y^{\alpha+2}\leq y^{2}+y^{\beta+2}$, true for every $y\geq 0$.

\end{itemize}

\paragraph{\textmd{\textit{Results}}}

We are now ready to apply our abstract theory to the semilinear wave equation (\ref{eqn:wave}). To be more precise, first we apply the results of section~\ref{sec:abstract} to strong solutions, and then we extend them to weak solutions. This can be done by approximation, because all bounds provided in section~\ref{sec:abstract} do not depend on the regularity of the solution, but just on the constants that appear in the assumptions.

\begin{prop}[Semilinear wave equation -- Universal bound/decay]\label{prop:wave}

The following statements apply to the semilinear wave equation (\ref{eqn:wave}) under the assumptions described above, and in particular in the regime $0<\alpha<\beta$, with $\beta$ satisfying (\ref{hp:beta-wave}).

\begin{enumerate}
\renewcommand{\labelenumi}{(\arabic{enumi})}

\item  In both Neumann  and Dirichlet cases, there exist two real numbers ${\Gamma}$ and $\Gamma_{\!*}$ such that any weak solution in $(0,T)$ satisfies
\begin{equation}
\|u_{t}\|_{L^{2}(\Omega)}^{2}+\|\nabla u\|_{L^{2}(\Omega)}^{2}+\|u\|_{L^{\beta+2}(\Omega)}^{\beta+2}\leq{\Gamma}\,t^{-1/\gamma}+\Gamma_{\!*}
\qquad
\forall t\in(0,T),
\label{th:wave-ubp}
\end{equation}
where $\gamma$ is defined by (\ref{defn:gamma}).

\item  In the Neumann case, if we assume that $\lambda=\mu=0$ and $h\equiv 0$, then there exists a real number ${D}$ such that any weak solution in $(0,+\infty)$ satisfies
\begin{equation}
\|u_{t}\|_{L^{2}(\Omega)}^{2}+\|\nabla u\|_{L^{2}(\Omega)}^{2}+\|u\|_{L^{\beta+2}(\Omega)}^{\beta+2}\leq{D}\,t^{-1/\gamma}
\qquad
\forall t\geq 1,
\label{th:wave-ud1}
\end{equation}
where $\gamma$ is defined by (\ref{defn:gamma-as}).

\item  In the Neumann case, if we assume that $\lambda<0$, $\mu=0$, and $h\equiv 0$, then there exists a real number ${D}$ such that any weak solution in $(0,+\infty)$ satisfies
\begin{equation}
\|u_{t}\|_{L^{2}(\Omega)}^{2}+\|u\|_{H^{1}(\Omega)}^{2}\leq{D}\,t^{-2/\alpha}
\qquad
\forall t\geq 1.
%\label{th:wave-ud2-n}
\nonumber
\end{equation}

\item  In the Dirichlet case, if we assume that $\lambda<\lambda_{1}(\Omega)$, $\mu=0$, and $h\equiv 0$, then there exists a real number ${D}$ such that any weak solution in $(0,+\infty)$ satisfies
\begin{equation}
\|u_{t}\|_{L^{2}(\Omega)}^{2}+\|\nabla u\|_{L^{2}(\Omega)}^{2}\leq{D}\,t^{-2/\alpha}
\qquad
\forall t\geq 1.
\label{th:wave-ud2-d}
\end{equation}

\end{enumerate}

\end{prop}

\begin{rmk}
\begin{em}
Let us comment on the bounds obtained in Proposition~\ref{prop:wave} above.
\begin{itemize}

\item Concerning (\ref{th:wave-ubp}), the bounds obtained for $\nabla u$ and $u_{t}$ are worse, as $t\to 0^{+}$, than the bound obtained for $u$. We do not know whether this corresponds to a real phenomenon or it is just due to our technique.

\item Concerning (\ref{th:wave-ud1}), the bounds obtained for $\nabla u$ and $u_{t}$ are better, as $t\to +\infty$, than the bound obtained for $u$. Again we do not know whether this corresponds to a real phenomenon or it is just due to our technique. One might think of a kind of homogeneization effet.

\end{itemize}

In general, the optimality of these decay rates is a challenging open problem. We refer to section~\ref{sec:open} for further comments.

\end{em}
\end{rmk}

%\clearpage

\subsection{Semilinear plate equations}%\label{sec:plate}

Let us consider, in a cylinder $(0,T)\times\Omega$ or $(0,+\infty)\times\Omega$, semilinear plate equations of the form
\begin{equation}
u_{tt}+\Delta^{2} u+b|u|^{\beta}u-\lambda u+c|u_{t}|^{\alpha}u_{t}-\mu u_{t}=h,
\label{eqn:plate}
\end{equation}
where $\alpha$, $\beta$, $b$, $c$ are positive real parameters, and $\lambda$, $\mu$ are real parameters. Let us add initial conditions (\ref{initial-data}), and either hinged boundary conditions
\begin{equation}
u(t,x)=\Delta u(t,x)=0
\qquad
\mbox{in }(0,T)\times\partial\Omega,
\nonumber
\end{equation}
or clamped boundary conditions
\begin{equation}
u(t,x)=\frac{\partial u}{\partial n}(t,x)=0
\qquad
\mbox{in }(0,T)\times\partial\Omega.
\nonumber
\end{equation}

If we assume that 
\begin{equation}
(N-4)\beta\leq 4,
\label{hp:beta-plate}
\end{equation}
then $H^{2}(\Omega)$ is continuously imbedded into $L^{2\beta+2}(\Omega)$. In this case the initial-boundary-value problem is classically well-posed for initial data in $(H^{2}(\Omega)\cap H^{1}_{0}(\Omega))\times L^{2}(\Omega)$ in the case of hinged boundary conditions, and for initial data in $H^{2}_{0}(\Omega)\times L^{2}(\Omega)$ in the case of clamped boundary conditions. Again ``well-posedness'' refers to weak solutions, while strong solutions exist under additional regularity assumptions on the initial data and on the forcing term $h$.

This problem fits in the abstract framework of section~\ref{sec:functional} if we define $g(t,v)$ as in (\ref{defn:gtv-wave}), we consider the spaces $H$, $X$, $Y$ as in (\ref{defn:HXY-wave}), and then we set
\begin{equation}
F(u):=\frac{1}{2}\|\Delta u\|_{L^{2}(\Omega)}^{2}-\frac{\lambda}{2}\|u\|_{L^{2}(\Omega)}^{2}+\frac{b}{\beta+2}\|u\|_{L^{\beta+2}(\Omega)}^{\beta+2},
\nonumber
\end{equation}
and $V:=H^{2}(\Omega)\cap H^{1}_{0}(\Omega)$ in the case of hinged boundary conditions, or $V:=H^{2}_{0}(\Omega)$ in the case of clamped boundary conditions.

The verification of the abstract assumptions of section~\ref{sec:functional} is analogous to the case of the semilinear wave equation. As a consequence, from Theorem~\ref{thm:ubp} and Theorem~\ref{thm:ud1} we obtain the following result (first for strong solutions, and then for weak solutions by a density argument).

\begin{prop}[Semilinear plate equation -- Universal bound/decay]

The following statements apply to the semilinear plate equation (\ref{eqn:plate}) under the assumptions described above, and in particular in the regime $0<\alpha<\beta$ with $\beta$ satisfying (\ref{hp:beta-plate}).

\begin{enumerate}
\renewcommand{\labelenumi}{(\arabic{enumi})}

\item  Both in the clamped and in the hinged case, there exist two real numbers ${\Gamma}$ and $\Gamma_{\!*}$ such that any weak solution in $(0,T)$ satisfies
\begin{equation}
\|u_{t}\|_{L^{2}(\Omega)}^{2}+\|\Delta u\|_{L^{2}(\Omega)}^{2}+\|u\|_{L^{\beta+2}(\Omega)}^{\beta+2}\leq{\Gamma}\,t^{-1/\gamma}+\Gamma_{\!*}
\qquad
\forall t\in(0,T),
\nonumber
\end{equation}
where $\gamma$ is defined by (\ref{defn:gamma}).

\item  Both in the clamped and in the hinged case, if we assume in addition that $\mu=0$, $\lambda<\lambda_{1}(\Omega)$ (where now $\lambda_{1}(\Omega)$ is the first eigenvalue of the bi-Laplacian with the corresponding boundary conditions), and $h\equiv 0$, then there exists a real number ${D}$ such that any weak solution in $(0,+\infty)$ satisfies
\begin{equation}
\|u_{t}\|_{L^{2}(\Omega)}^{2}+\|u\|_{H^{2}(\Omega)}^{2}\leq{D}\,t^{-2/\alpha}
\qquad
\forall t\geq 1.
\nonumber
\end{equation}

\end{enumerate}

\end{prop}

%\clearpage

\subsection{Quasilinear equations of Kirchhoff type} 

Let us consider, in a cylinder $(0,T)\times\Omega$ or $(0,+\infty)\times\Omega$, quasilinear integro-differential equations with averaged damping of the form
\begin{equation}
u_{tt}-\Delta u -b\left(\int_\Omega |\nabla u|^2\,dx\right )^{\beta/2}\Delta u + c\left(\int_\Omega |u_{t}|^2\, dx\right)^{\alpha/2} u_{t} -\lambda u-\mu u_{t}=h,
\label{eqn:kirchhoff}
\end{equation}
where $\alpha$, $\beta$, $b$, $c$ are positive real parameters, and $\lambda$, $\mu$ are real parameters. Let us add initial conditions (\ref{initial-data}), and homogeneous Dirichlet boundary conditions (Neumann boundary conditions are not allowed in this case, as shown in Remark~\ref{rmk:n-kirchhoff} below).

This problem fits in the form of (\ref{eqn:abstract}) if we set
\begin{equation}
F(u):=\frac{1}{2}\|\nabla u\|_{L^{2}(\Omega)}^{2}+\frac{b}{\beta+2}\|\nabla u\|_{L^{2}(\Omega)}^{\beta+2}-\frac{\lambda}{2}\|u\|_{L^{2}(\Omega)}^{2},
\label{defn:F-kirchhoff}
\end{equation}
and
\begin{equation}
[g(t,v)](t,x):=c\|v\|_{L^{2}(\Omega)}^{\alpha} v(x)-\mu v(x)+h(t,x).
\nonumber
\end{equation}

As for the functional spaces, we choose $H=X=L^{2}(\Omega)$, $V:=H^{1}_{0}(\Omega)$, and $Y$ any space between $V$ and $H$, endpoints included. The verification of the abstract assumptions of section~\ref{sec:functional} is similar, and sometimes simpler, to the case of the semilinear wave equation. We leave the details to the interested reader.

Existence of weak or strong global solutions to these equations is a big open problem, and it is known to be true only in special cases, for example when both the initial data and the forcing term are analytic and satisfy suitable compatibility conditions. We refer to~\cite{GG:k-decay,GG:SNS,GG:k-Nishihara} for further details. To remain on the safe side, we can assume that both  initial state $(u_{0},u_{1})$ and forcing term $h(t,x)$ are finite linear combinations of eigenfunctions of the Dirichlet Laplacian. In this case the problem is equivalent to a finite system of ordinary differential equations, and existence of global strong solutions in the sense of Definition~\ref{defn:strongsol} is substantially trivial. 

In any case the functions $F$ and $g$ fit in the abstract framework of  section~\ref{sec:functional}, and therefore all (weak or strong) solutions, provided they exist, satisfy the universal bound/decay properties in the energy space, as follows.

\begin{prop}[Quasilinear Kirchhoff equation -- Universal bound/decay]

The following statements apply to the quasilinear Kirchhoff equation (\ref{eqn:kirchhoff}) under the assumptions described above, and in particular in the regime $0<\alpha<\beta$ with Dirichlet boundary conditions.

\begin{enumerate}
\renewcommand{\labelenumi}{(\arabic{enumi})}

\item   There exist two real numbers ${\Gamma}$ and $\Gamma_{\!*}$ such that any weak solution in $(0,T)$ satisfies
\begin{equation}
\|u_{t}\|_{L^{2}(\Omega)}^{2}+\|\nabla u\|_{L^{2}(\Omega)}^{\beta+2}\leq{\Gamma}\,t^{-1/\gamma}+\Gamma_{\!*}
\qquad
\forall t\in(0,T),
\nonumber
\end{equation}
where $\gamma$ is defined by (\ref{defn:gamma}).

\item  If we assume in addition that $\mu=0$, $\lambda<\lambda_{1}(\Omega)$ (where $\lambda_{1}(\Omega)$ is the first eigenvalue of the Dirichlet Laplacian), and $h\equiv 0$, then there exists a real number ${D}$ such that any weak solution in $(0,+\infty)$ satisfies
\begin{equation}
\|u_{t}\|_{L^{2}(\Omega)}^{2}+\|\nabla u\|_{L^{2}(\Omega)}^{2}\leq{D}\,t^{-2/\alpha}
\qquad
\forall t\geq 1.
\nonumber
\end{equation}

\end{enumerate}

\end{prop}

\begin{rmk}\label{rmk:n-kirchhoff}
\begin{em}
The universal bound/decay properties do not hold for (\ref{eqn:kirchhoff}) with homogeneous Neumann boundary conditions. This is evident when $\lambda=0$ and $h\equiv 0$, in which case all constant functions are solutions. From the technical point of view, we observe that in the Neumann case there is no choice of the function space $Y$ that guarantees that the function $F(u)$ defined by (\ref{defn:F-kirchhoff}) satisfies (\ref{hp:F2}), because there is no way to control $u$ in terms of $\nabla u$.
\end{em}
\end{rmk}

\begin{rmk}
\begin{em}

Several variants of (\ref{eqn:kirchhoff}) fit into our abstract framework. We just mention the degenerate hyperbolic equation
\begin{equation}
\label{wave-Kircheasy}
u_{tt}  -\left (\int _\Omega |\nabla u|^2 dx\right )^{\beta/2}\Delta u +  \left (\int _\Omega |u_{t}|^2 dx\right )^{\alpha/2}u_{t}=0, 
\end{equation}
even with ``local'' damping
\begin{equation}
%\label{wave-Kirch}
u_{tt}  - \left(\int _\Omega |\nabla u|^2 dx\right )^{\beta/2}\Delta u + |u_{t}|^{\alpha}u_{t}=0.
\nonumber 
\end{equation} 

Existence of global solutions is a widely open problem, if we exclude very special cases such as finite linear combinations of eigenfunctions for the first equation (see~\cite{1996-JMAA-MatIke} for a non-degenerate equation with local nonlinear dissipation). For all these equations our theory provides the universal bound and the universal decay for all solutions that are proved to exist and do satisfy the nergy identities.

\end{em}
\end{rmk}
  
%\clearpage

\section {Negative results and open problems}\label{sec:open}

We already know that the universal bound property of Theorem~\ref{thm:ubp} may fail when $0\leq\beta\leq\alpha$. The basic example is the result in~\cite{1998-DIE-Souplet} concerning the ordinary differential equation (\ref{ode:souplet}). This result has some simple spin-offs in terms of partial differential equations, for example
\begin{itemize}

\item  in the case of semilinear wave equations of the form (\ref{wave})
with homogeneous Neumann boundary conditions in a bounded open set, with initial data that do not depend on space variables (indeed also the solution does not depend on space variables, and the equation reduces to (\ref{ode:souplet})),

\item  in the case of the Kirchhoff equation (\ref{wave-Kircheasy}) with Dirichlet boundary conditions and initial data that are multiples of the same eigenfunction of the Dirichlet Laplacian (indeed the solution is the product of the eigenfunction itself with a solution to a scalar ordinary differential equation of the form (\ref{ode:souplet})). 

\end{itemize}

A less trivial example where the universal bound property is known to fail is the case of equation
\begin{equation}
u_{tt} -\Delta u +|u_{t}|^{\alpha}u_{t}=0
\label{eqn:nl-damping}
\end{equation}
with homogeneous Dirichlet boundary conditions in a bounded open set with smooth boundary (see~\cite[Theorem~1]{1994-JMPA-Carpio}). This example motivates the following question, whose answer does not seem to follow from any result presently established.

\begin{open}
\begin{em}

Does there exist a counterexample to the uniform bound/decay property for solutions to equation (\ref{wave}) with Dirichlet boundary conditions, of course in the regime $0\leq\beta\leq\alpha$?

\end{em}
\end{open}

Another interesting question concerns the optimality of the decay estimates, for example in the case of the semilinear wave equation. In the case of Neumann boundary conditions, the consideration of solutions that are constant in space reveals that the estimates for $u$ and $u_{t}$ provided by (\ref{th:wave-ud1}) are optimal (it is the trivial sense of the energy inequality). The optimality of the estimates on $\nabla u$, as well as the case of Dirichlet boundary conditions, remains a challenging problem, even in one space dimension.

\begin{open}
\begin{em}

Are the estimates on $|\nabla u|$ provided by (\ref{th:wave-ud1}) optimal?
Are the estimates on $u$, $u_{t}$ and $|\nabla u|$ provided by (\ref{th:wave-ud2-d}) optimal?

\end{em}
\end{open}

The last question above is connected to the classical problem concerning the decay rate of individual solutions to the simpler equation (\ref{eqn:nl-damping}). We refer to~\cite[Problem 4.1]{HPr} for further details on this problem.

%\clearpage

\subsubsection*{\centering Acknowledgments}

The first two authors are members of the ``Gruppo Nazionale per l'Analisi Matematica, la Probabilità e le loro Applicazioni'' (GNAMPA) of the ``Istituto Nazionale di Alta Matematica'' (INdAM). 

%\clearpage

%\bibliographystyle{MaxNew}
%\bibliography{../../../BibTeX/UniversalBound}

%\label{NumeroPagine}

\end{document}